\documentclass{article}

\usepackage{hyperref}
\usepackage[english]{babel}
\usepackage[utf8]{inputenc}
\usepackage[T1]{fontenc}

\usepackage{mathtools,amsthm,amsfonts,amssymb,bm}
\usepackage[margin=3cm]{geometry}
\usepackage{graphicx}
\usepackage{dsfont}
\usepackage{mathrsfs}
\usepackage{csquotes}
\usepackage[ruled,vlined]{algorithm2e}
\usepackage{subcaption}
\usepackage{multirow}

\theoremstyle{definition}
\newtheorem{definition}{Definition}[section]

\theoremstyle{plain}
\newtheorem{theorem}{Theorem}[section]
\newtheorem{proposition}{Proposition}[section]
\newtheorem{corollary}{Corollary}[section]

\def\NN{{\mathbb N}} 
\def\PP{{\mathbb P}}
\def\RR{{\mathbb R}}

\def\II{{\mathds 1}}

\newcommand{\new}[1]{#1}

\title{Maximum Likelihood Estimation for Hawkes Processes with self-excitation or inhibition}
\author{Anna Bonnet, Miguel Martinez Herrera, Maxime Sangnier\footnote{Sorbonne Université, CNRS, LPSM, Paris, France}}

\begin{document}

\maketitle

\begin{abstract}

In this paper, we present a maximum likelihood method for estimating the parameters of a univariate Hawkes process with self-excitation or inhibition. Our work generalizes techniques and results that were restricted to the self-exciting scenario.
The proposed estimator is implemented for the classical exponential kernel and we show that, in the inhibition context, our procedure provides more accurate estimations than current alternative approaches.
\end{abstract}

\section{Introduction}

The Hawkes model is a point process observed on the real line, which generally corresponds to the time, where any previously encountered event has a direct influence on the chances of future events occurring. This past-dependent mathematical model was introduced in \cite{Hawkes1971} and its first application was to model earthquakes occurrences \cite{Ogata1988, Ogata1998}. Since then, Hawkes processes have been widely used in various fields, for instance finance \cite{Bacry2013}, social media \cite{Rizoiu2017, Mishra2016}, epidemiology \cite{Rizoiu2018}, sociology \cite{Linderman2014} and neuroscience \cite{Reynaud2014}.

The main advantage of Hawkes processes is their ability to model different kinds of relationships between phenomena through an unknown kernel or transfer function. The Hawkes model was originally introduced as a self-exciting point process where the appearance of an event increases the chances of another one triggering.
Several estimation procedures have been proposed for the kernel function, both in parametric \cite{Ogata1988, DaFonseca2013, Ozaki1979} and nonparametric \cite{Reynaud2014, Bacry2015} frameworks.

However, the inhibition setting, where the presence of an event decreases the chance of another occurring, has drawn less attention in the literature, although it can be of great interest in several fields, in particular in neuroscience \cite{Reynaud2018}. In this inhibition context, the cluster representation \cite{Hawkes1974} on which is based the construction of a self-exciting Hawkes process, is no longer valid.
While the existence and the construction of such nonlinear processes can be found in recent works for the univariate \cite{Costa2018} and multivariate \cite{Chen2017} cases,
statistical estimation of the kernel function has been hardly addressed.
A first approach consists in computing an approximation of the likelihood as if the intensity function could take negative values, and optimizing it to get a maximum likelihood estimator \cite{Lemonnier2014}.
\new{Alternatively,} the type of interaction (excitation or inhibition) can be considered as a hidden variable, giving rise to a very practical estimation method \cite{Eisner2017}.

In this paper, we propose a maximum likelihood procedure that can handle both excitation and inhibition scenarios for a univariate Hawkes process. Our approach is based on an explicit computation of the likelihood for any type of \new{monotone} kernel functions, which is facilitated by the introduction of the natural concept of restart points. The \new{latter} are the times when the intensity function, that can be null on some intervals, become strictly positive again. We show that these restart points have a closed-form expression when the kernel is exponential, which allows us to rewrite and maximize the likelihood without approximations that are proposed for instance in \cite{Lemonnier2014}.  Our estimator is implemented in Python (the code is freely available online\footnote{\url{https://github.com/migmtz/hawkes-inhibition-expon}}). We also propose a numerical study which shows the good performance of our exact estimation procedure compared to approximated approaches, especially when the intensity function is frequently equal to zero.

To outline the paper, besides a quick introduction to self-regulating Hawkes processes (\new{also referred to as self-correcting Hawkes processes or Hawkes processes with inhibition}), Section~\ref{sec:general} introduces the concepts of underlying intensity function and restart points.
General results concerning the compensator and the exact maximum likelihood estimation procedure are described in Section~\ref{sec:mle}.
At last, after a brief discussion about goodness-of-fit in Section~\ref{sec:goodness},
Section~\ref{sec:num_res} concludes with a numerical study of the estimation error.

\section{The Hawkes process}\label{sec:general}

Let $N$ be a point process on $\RR_+^*$, \new{where \(\RR_+^* = \left\{ x>0 : x \in \RR \right\}\),} and $(T_k)_{k\geq 1}$ its associated event times (with convention $T_0 = 0$).
For any $t \geq 0$, let us note $N(t) = \sum_{k\geq 1}{\II_{T_k \leq t }}$ the number of events in $[0,t]$
\new{(where \(\II_{\cdot}\) stands for the indicator function)},
and $\lambda$ its conditional intensity function \cite{DaleyV1}:
\begin{equation*}
	\lambda(t) = \lim_{h\to 0}{\frac{\PP(N(t+h) - N(t) > 0)}{h}}.
\end{equation*}

A univariate Hawkes process is a point process defined by the conditional intensity function:
\begin{align}
    \lambda(t) &= \left(\lambda_0 + \int_{\new{0}}^{t}{h(t-s)\,\mathrm{d}N(s)}\right)^+
    = \left(\lambda_0 + \sum_{T_k \leq t}{h(t-T_k)}\right)^+,
    \label{eq:general_hawkes}
\end{align}
where $x^+ = \max(0,x)$ denotes the positive part of any real value $x$,
$\lambda_0\in\RR_+^*$ is the baseline intensity and
\new{\(h : \RR_+ \to \RR\)} is the kernel, which is assumed to be a monotone measurable function with $\lim_{t\to\infty} h(t) = 0$.
The kernel function $h$ is the key component of a Hawkes process:
it translates the influence (generally assumed to fade away over time) of a past event over the process.
Here, $h$ is allowed to take negative values, meaning that it can model both self-exciting and self-regulating Hawkes processes.

Working with such Hawkes processes may prove to be difficult as the positive part function is non-linear.
In particular, while computing the compensator function \cite{DaleyV1}
\begin{equation}
	\Lambda(t) = \int_{0}^{t}{\lambda(t)\,\mathrm{d}t},
	\quad \forall t \geq 0,
    \label{eq:compensator}
\end{equation}
is very easy in the self-exciting case (by linearity of the intensity), it becomes more challenging for the self-regulating Hawkes process.
As it is the keystone to derive the likelihood function (and then to obtain a parametric estimation method), our first contribution is to provide an exact expression of the compensator.

For this purpose, let us first introduce the \emph{underlying intensity function} and the \emph{restart time}, two quantities which will allow us to
\new{derive the computation of the likelihood of a monotone Hawkes process, in a framework unifying self-correcting and self-exciting Hawkes processes.}

\begin{definition}
Let the \emph{underlying intensity function} of $N$ be:
\begin{equation*}
    \lambda^\star(t) = \lambda_0 + \int_{\new{0}}^{t}{h(t-s)\,\mathrm{d}N(s)}.
\end{equation*}

In addition, let the \emph{restart time} $T_k^\star$ be, for any positive integer $k$:
\begin{equation*}
    T_k^\star = \inf{\{t\geq T_k\mid \lambda(t) > 0\}},
\end{equation*}
along with its corresponding \emph{cooldown interval} $\tau_k^* = T_k^* - T_k$.
\label{def:underlying}
\end{definition}

\begin{figure}[ht]
  \centering
  \includegraphics[width=0.7\textwidth]{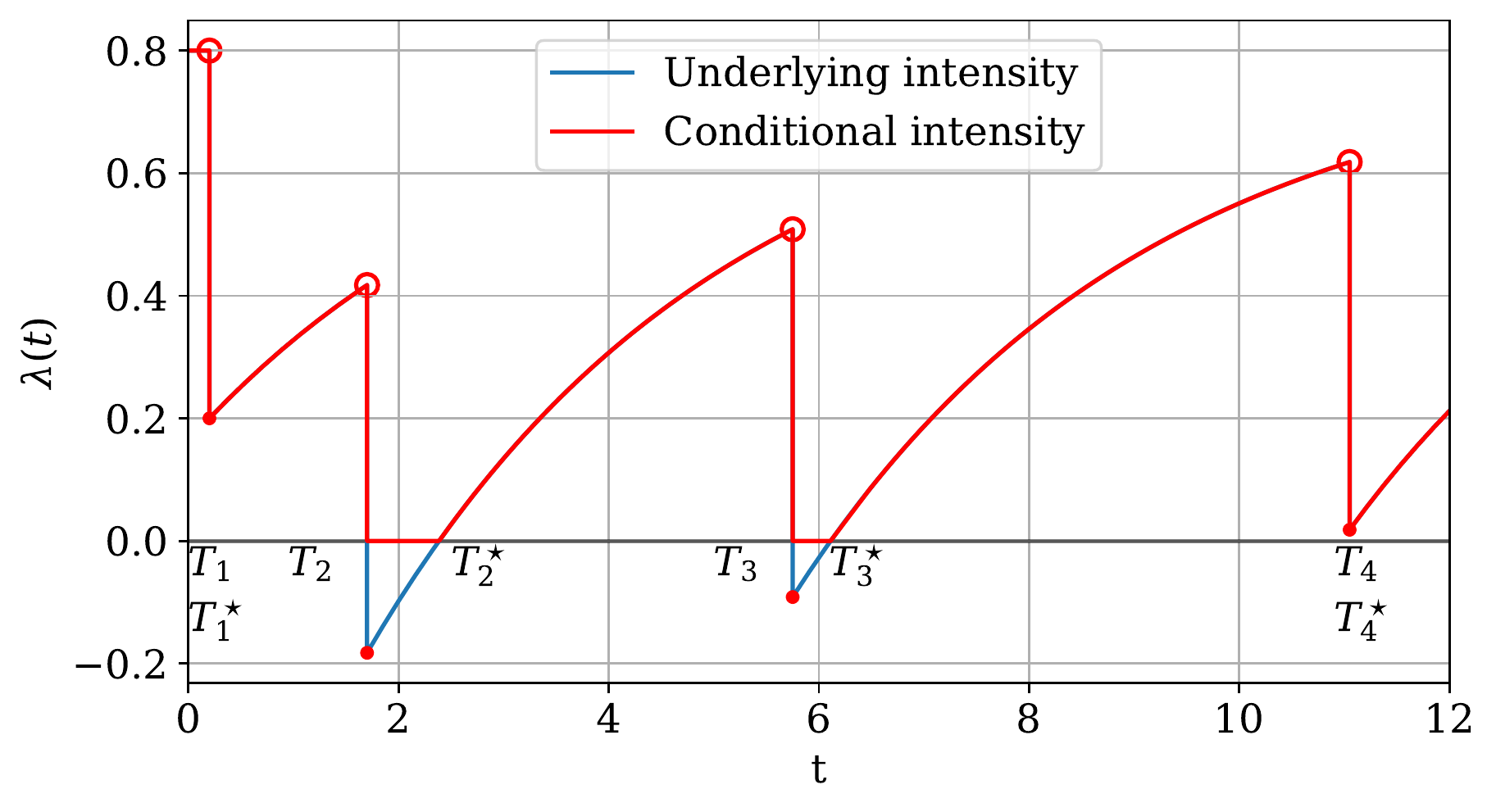}  
  \caption{Example of the intensity (red curve) and underlying intensity (blue curve) for a self-regulating Hawkes process, with the associated restart times. We only see the negative values of the blue curve since they precisely correspond to the values for which the two intensity functions are not equal.
  }
  \label{fig:underlying_intensity}
\end{figure}

As illustrated in Figure \ref{fig:underlying_intensity}, $\lambda^\star$ corresponds to the intensity $\lambda$ as if it were allowed to take negative values.
Moreover, as the kernel is assumed to be monotone, the restart time associated to one occurrence can be interpreted as the first moment after this occurrence from which $\lambda$ and $\lambda^\star$ become equal  (in particular, the restart time and the occurrence time coincide if the intensity function is nonnegative at this time, see Figure \ref{fig:underlying_intensity}):
\begin{equation*}
    T_k^\star = \inf{\{t\geq T_i\mid \forall t\in(T_k^*,T_{k+1}),\, \lambda(t) = \lambda^\star(t)\}}.
\end{equation*}

\section{Maximum likelihood estimation and the exponential model}
\label{sec:mle}

Assume a parametric model $\mathcal P = \left\{ \lambda_\theta, \theta \in \Theta \right\}$ for the conditional intensity function $\lambda$, where $\theta$ contains unknown quantities such as the baseline $\lambda_0$ and the kernel $h$.
Then, with convention $log(t) = -\infty$ for $t\leq0$, the log-likelihood $\ell_t$ of any $\theta \in \Theta$ with respect to the observations $T_1,\ldots, T_{N(t)}$ in the time interval $[0,t]$ is \cite[Proposition 7.2.III.]{DaleyV1}, \cite{Ozaki1979}:
\begin{equation}\label{eq:log_likelihood}
    \ell_t(\theta) = \sum_{k=1}^{N(t)}{\log{(\lambda_\theta(T_k^-)})} - \Lambda_\theta(t),
\end{equation}
where the compensator $\Lambda_\theta$ is defined as in Equation \eqref{eq:compensator} and $\lambda_\theta(T_k^-) = \lim_{t \to T_k^-} \lambda_\theta(t)$.

Equation~\eqref{eq:log_likelihood} reveals the importance of being able to compute the compensator $\Lambda$ (equivalently $\Lambda_\theta$) in order to provide a practical implementation of the maximum likelihood estimator of $\lambda$.
Thus, a first contribution of this paper lies in Proposition~\ref{proposition:integral_cooldown}, which establishes a decomposition of the compensator $\Lambda$ using the underlying intensity function $\lambda^\star$ and the restart times $T_{1}^\star,\ldots, T_{N(t)}^\star$.

\begin{proposition}\label{proposition:integral_cooldown}
For any $t>0$, the compensator $\Lambda$ can be expressed as:

\begin{equation}\label{eq:compensator_general}
\Lambda(t) =
\begin{dcases}
    \lambda_0 t &\text{\qquad if $t<T_1$}\\
    \lambda_0 T_1 + \sum_{k=1}^{{N(t)-1}}{\int_{T_{k}^\star}^{T_{k+1}}{\lambda^\star(u)\,\mathrm{d}u}} + \int_{T_{N(t)}^\star}^{t}{\lambda^\star(u)\,\mathrm{d}u} &\text{\qquad if $t\geq T_1$},
\end{dcases}
\end{equation}
with the conventions that the sum is equal to $0$ if ${N(t)} = 1$ and the last integral is equal to $0$ if $t < T_{N(t)}^*$.
\end{proposition}
\begin{proof}
This comes directly from splitting the integral of $\Lambda(t) = \int_{0}^{t}{\lambda(t)\,\mathrm{d}t}$ on the intervals $[T_k, T_{k+1})$ ($k \in \{ 0, \dots, N(t)-1 \}$) and $[T_{N(t)}, t]$, and by remarking that, since $h$ is monotone, $\forall t\in[T_k, T_{k+1})$, $\lambda(t) = \lambda^\star(t)\II_{[T_k^*,T_{k+1})}(t)$.

\end{proof}

In order to give an explicit computation of the quantity $\int_{T_{k}^\star}^{T_{k+1}}{\lambda^\star(u)\,\mathrm{d}u}$ (equivalently $\int_{T_{N(t)}^\star}^{t}{\lambda^\star(u)\,\mathrm{d}u}$) which appears in Proposition~\ref{proposition:integral_cooldown}, we focus on the classical scenario where we consider an exponential kernel $h(t) = \alpha \mathrm{e}^{-\beta t}$, for some $\alpha\in\RR$ and $\beta\in\RR_+^*$.
Let us notice that $\alpha$ can be either positive or negative, meaning that the process may be either self-exciting or self-regulating.

Then, the underlying intensity function can be written as:
\begin{equation}\label{eq:exponential_intensity}
    \lambda^\star(t) = \lambda_0 + \int_{\new{0}}^{t}{\alpha \mathrm{e}^{-\beta(t-s)}\,\mathrm{d}N(s)}.
\end{equation}
The forthcoming proposition steps forward in computing the compensator for an exponential kernel.

\begin{proposition}[Compensator for exponential kernel]
Let $t > 0$ and $k \in \{1, \dots, N(t)\}$.
The restart times read:
\[
	T_k^\star = T_k + \beta^{-1} \log \left({\frac{\lambda_0 - \lambda^\star(T_k)}{\lambda_0}} \right) \new{\II_{\lambda^\star(T_k) < 0}},
\]
and the compensator is expressed as in Equation~\eqref{eq:compensator_general}, with,
for any $\tau \in [T_{k}^*, T_{k+1}]$:
\[
	\int_{T_{k}^\star}^{\tau}{\lambda^\star(u)\,\mathrm{d}u}
	= \lambda_0(\tau - T_{k}^\star) + \beta^{-1} (\lambda^\star(T_{k}) - \lambda_0) (\mathrm{e}^{-\beta(T_{k}^\star-T_{k})}-\mathrm{e}^{-\beta(\tau-T_{k})}).
\]

\label{prop:exp}
\end{proposition}
\begin{proof}
The proof is in \ref{app:appendix}.
\end{proof}

\begin{corollary}[Log-likelihood for exponential kernel]
  Let
  \begin{equation}
    \mathcal P = \left\{ \lambda_\theta = \bar{\lambda_0} + \int_{\new{0}}^{t}{\bar \alpha \mathrm{e}^{-\bar \beta(t-s)}\,\mathrm{d}N(s)} : \theta=(\bar{\lambda_0}, \bar \alpha, \bar \beta) \in \Theta \right\},
    \label{equ:model}
  \end{equation}
  be a parametric exponential model for the conditional intensity function $\lambda$ with $\Theta = \RR_+^* \times \RR \times \RR_+^*$,
  along with the candidate compensator $\Lambda_\theta$, the underlying intensity function $\lambda_\theta^\star$ and the restart times $T_{\theta, 1}^\star,\ldots, T_{\theta, N(t)}^\star$ associated to $\lambda_\theta$ (see Equation~\eqref{eq:compensator} and Definition~\ref{def:underlying}). 

  For any $\theta=(\bar{\lambda_0}, \bar \alpha, \bar \beta) \in \Theta$, by denoting
  \[
    \Lambda_{\theta, k} =
    \bar \lambda_0(T_k - T_{\theta, k-1}^\star) + \bar \beta^{-1} (\lambda_\theta^\star(T_{k-1}) - \bar \lambda_0) (\mathrm{e}^{-\bar \beta(T_{\theta, k-1}^\star-T_{k-1})}-\mathrm{e}^{-\bar \beta(T_k-T_{k-1})}),
  \]
  the log-likelihood reads (with convention $\log(x) = -\infty$ for $x \leq 0$):
  \begin{align}
      \ell_t(\theta)
      =& \log{\bar{\lambda_0}}
      - \bar{\lambda_0} T_1
      + \sum_{k=2}^{N(t)} \left[ \log \left(\bar{\lambda_0} + (\lambda_\theta^\star(T_{k-1}) - \bar{\lambda_0}) \mathrm{e}^{-\bar \beta(T_{k}-T_{k-1})} \right) - \Lambda_{\theta, k} \right] \notag\\
      &- \left[ \bar{\lambda_0} (t - T_{\theta, N(t)}^\star) + \bar \beta^{-1} (\lambda_\theta^\star(T_{N(t)}) - \bar{\lambda_0}) \left(\mathrm{e}^{-\bar \beta(T_{\theta, N(t)}^\star-T_{N(t)})}-\mathrm{e}^{-\bar \beta(t-T_{N(t)})} \right) \right] \II_{t > T_{\theta, N(t)}^\star}
      \label{eq:likeli_exp}.
  \end{align}
  \label{cor:loglik}
\end{corollary}
\begin{proof}
  By Equation~\eqref{eq:markov_expression} in the proof of Proposition~\ref{prop:exp},
  \begin{equation*}
  \lambda_\theta^\star(T_k^-) =
  \begin{dcases}
      \bar \lambda_0 &\text{\qquad if $k = 1$,}\\
      \bar \lambda_0 + (\lambda_\theta^\star(T_{k-1}) - \bar \lambda_0) \mathrm{e}^{-\bar \beta(T_{k}-T_{k-1})} &\text{\qquad if $k\geq2$.}
  \end{dcases}
\end{equation*}
Combining this expression with Propositions \ref{proposition:integral_cooldown} and \ref{prop:exp} leads to the result.
\end{proof}

Corollary~\ref{cor:loglik} exhibits that the log-likelihood for
self-regulating Hawkes processes with an exponential kernel can be evaluated in $O(N(t))$ operations (by computing iteratively the quantities $T_{\theta, k}^\star$ and $\Lambda_{\theta, k}$ appearing in the summation of Equation~\eqref{eq:likeli_exp}),
as already known
for self-exciting exponential Hawkes processes \cite[Chapter 4.2]{Laub2014}.
For other \new{monotone} kernels without the Markov property, evaluating the log-likelihood with the method proposed here requires $O(N(t)^2)$ operations, similarly to existing approaches for self-exciting Hawkes processes.

\section{Goodness-of-fit}
\label{sec:goodness}

Even though computing the compensator $\Lambda$ (equivalently $\Lambda_\theta$) was clearly motivated by maximum likelihood estimation, it turns out that it is of great benefit to assess goodness-of-fit, and in particular to check the validity of a maximum likelihood estimation.
This is possible thanks to the Time Change Theorem, a result originally stated for inhomogeneous Poisson processes.

\begin{theorem}[{\cite[Theorem 7.4.IV]{DaleyV1}}]\label{th:real_time_change}
  Assume that $\Lambda$ is continuous, monotone and $\Lambda(t)\xrightarrow[t\rightarrow \infty]{} \infty$ a.s.
  Then a.s., a sequence of event times $(U_k)_{k\ge 1}$ is a realization of $N$ if and only if $(\Lambda(U_k))_{k \ge 1}$ is a realization of a homogeneous Poisson process with unit intensity.
\end{theorem}

\new{Let us note that we can find applications of Theorem~\ref{th:real_time_change} to self-exciting Hawkes processes in the literature \cite[Chapter 5]{Laub2014}.
Since for self-regulating Hawkes processes \(\Lambda\) is still monotone, this result can also be applied in our case.

To be more precise, let us consider \(\theta \in \Theta\) and the null hypothesis: ``$(U_k)_{k\ge 1}$ is a realization of an exponential Hawkes process with parameter \(\theta\)''.
This hypothesis can be tested by applying a Kolmogorov-Smirnov test between the empirical distribution of $\left( \Lambda_\theta(U_{k+1}) - \Lambda_\theta(U_k) \right)_{k \ge 1}$ and an exponential distribution with parameter \(1\).
This procedure is illustrated in Table~\ref{table:estim_pvalue}, Section~\ref{sec:num_res}.}

\section{Numerical Results}\label{sec:num_res}

This section is aimed at assessing the maximum likelihood estimation method for self-regulating Hawkes processes, based on the exact computation of the compensator \(\Lambda_\theta\) in the exponential model \eqref{equ:model} (Corollary~\ref{cor:loglik}).
This procedure is compared to the approximated maximum likelihood estimation proposed in \cite{Lemonnier2014}, which consists in approximating $\Lambda_\theta$ by:
\[
  \Lambda_\theta^{LM}(t) = \int_{0}^{t}{\lambda_\theta^\star(u)\,\mathrm{d}u}.
\]
\new{This optimization procedure is performed with the L-BFGS-B algorithm from the Scipy package (with $(1,0,1)$ as a starting guess and a bounds argument such that $\lambda_0 \geq 0, \alpha \in \RR, \beta\geq 0$).}
In other words, estimators are:
\[
  \new{
  \hat \theta \in \operatorname{arg\,max}_{\theta \in \Theta} ~
  \left\{ \ell_{T_{N_{max}}}(\theta)
  =
  \sum_{k=1}^{N_{max}}{\log{(\lambda_\theta(T_k^-)})} - \Lambda_\theta(T_{N_{max}}) \right\},
  }
\]
where
\new{\(N_{max}=200\)} is the total number of jumps
and \(\Lambda_\theta\) can be replaced by \(\Lambda_\theta^{LM}\) to obtain the approximated likelihood proposed in \cite{Lemonnier2014}.

The comparison between the exact and the approximated estimation procedure is based on simulated data sets coming from self-correcting Hawkes processes of the form \eqref{equ:model} with \new{6} different values of \(\theta = (\bar \lambda_0, \bar \alpha, \bar \beta) \in \Theta\) \new{(see Table~\ref{table:estim_pvalue})} which have been chosen in order to explore different scenarios, in particular depending on whether the intensity function is frequently null or not.
Observations are sets of time jumps generated with a sampling algorithm (see the algorithm in \ref{app:simulation} and Python implementation online), which is a particular case of Ogata's thinning simulation method \cite{Ogata1981} that can handle Hawkes processes with either self-excitation or inhibition.

\begin{figure}[ht]
  \centering
    \includegraphics[width=\textwidth]{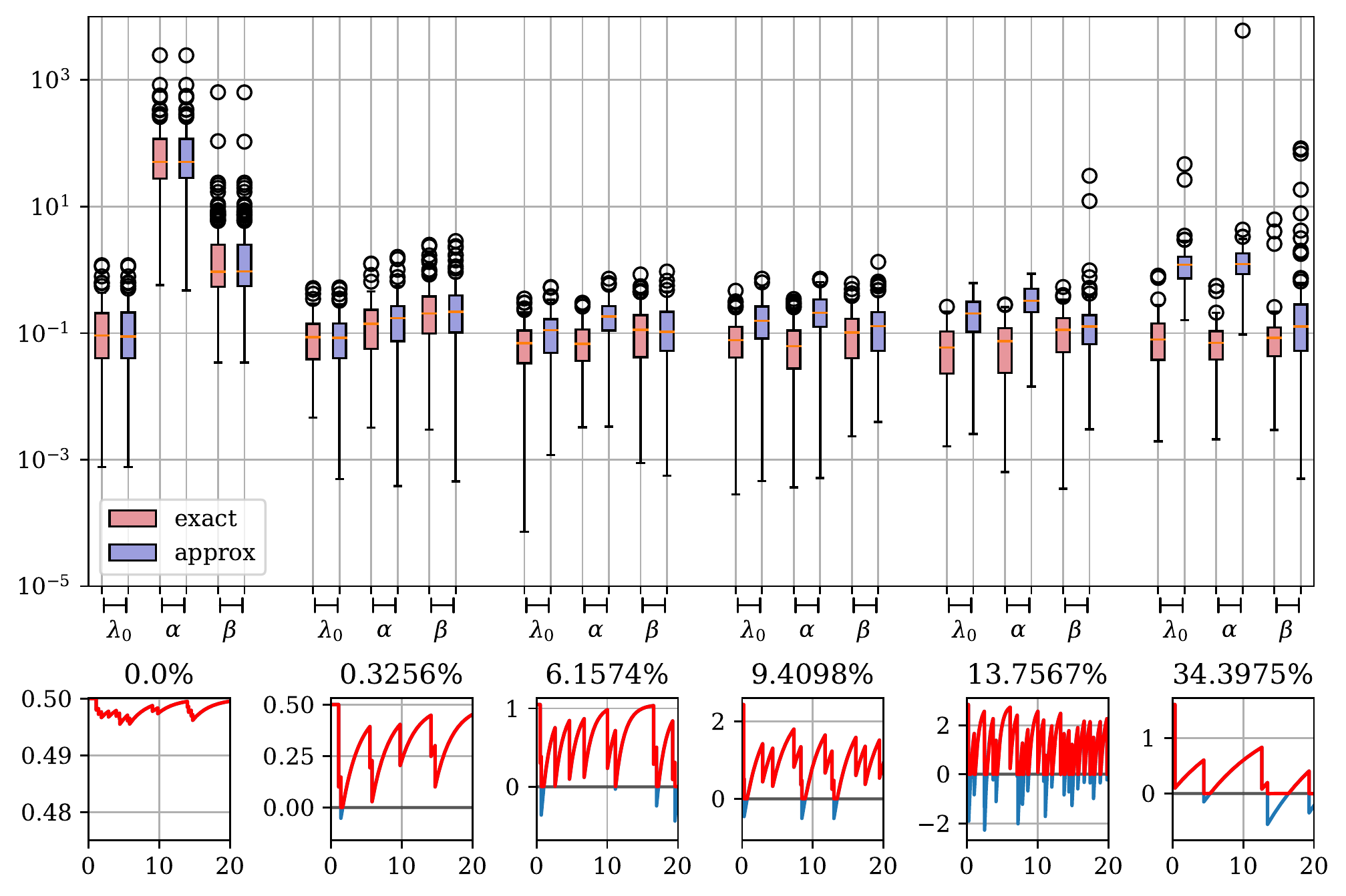}
  \caption{Top panel: \new{relative} absolute errors of estimations \(\hat \theta = (\hat \lambda_0, \hat \alpha, \hat \beta)\).
  Bottom panel: example of simulated intensities for each set of values $ \theta = ( \bar \lambda_0,  \bar \alpha,  \bar \beta)$  with the corresponding average percentage of time when the intensities are equal to zero.
  }
  \label{fig:boxplot_intensity}
\end{figure}

Figure~\ref{fig:boxplot_intensity} represents the \new{relative} absolute errors of estimations \(\hat \theta = (\hat \lambda_0, \hat \alpha, \hat \beta)\) for each of the \new{6} simulated models.
We observe that the exact approach provides more accurate estimations than the approximated procedure \new{(as illustrated in the boxplots of Figure~\ref{fig:boxplot_intensity} and by the \(p\)-values of the goodness-of-fit tests in Table~\ref{table:estim_pvalue})}. As expected, the more time the conditional intensity equals 0 (from left to right in Figure~\ref{fig:boxplot_intensity}), the greater the differences between the two procedures.
\new{Furthermore, the leftmost boxplot confirms that when the underlying intensity is nonnegative both methods are mostly identical.
Let us note that in this case the estimation of \(\bar \alpha\) is rather wrong (the estimation of \(\bar \beta\) is impacted consequently) probably because its value is close to \(0\) compared to the magnitude of \(\bar \lambda_0\).}

\begin{table}[!h]
\centering
\begin{tabular}{cccc|ccc|c}
\hline
	& \multicolumn{3}{c}{Parameters} & \multicolumn{3}{c}{Estimations} & \\
   & $\bar \lambda_0$ &  $\bar \alpha$ &  $\bar \beta$ & $\hat \lambda_0$ &  $\hat \alpha$ &  $\hat \beta$ &  p-value \\
\hline
       Exact&\multirow{2}{*}{0.5 } & \multirow{2}{*}{-0.001 } & \multirow{2}{*}{0.4} &       0.53 & 0.05 & 4.25 & 0.78\\
       Approx&		& 		&		& 0.54 & 0.05 & 4.23 & 0.78\\ \hline
       Exact&\multirow{2}{*}{0.5 } & \multirow{2}{*}{-0.2 } & \multirow{2}{*}{0.4 } &       0.52 & -0.21 & 0.42  & 0.72    \\
       Approx&		& 		&		&0.52 & -0.22 & 0.44 & 0.70\\ \hline
       Exact&\multirow{2}{*}{1.05 } & \multirow{2}{*}{-0.75 } & \multirow{2}{*}{0.8 } &       1.06 & -0.76 & 0.80 & 0.69  \\
       Approx&		& 		&		&1.14 & -0.88 & 0.82 & 0.55\\ \hline
       Exact&\multirow{2}{*}{2.43 } & \multirow{2}{*}{-0.98 } & \multirow{2}{*}{0.4} &       2.55 & -1.01 & 0.39 & 0.73  \\
       Approx&		& 		&		& 2.83 & -1.22 & 0.42 & 0.51 \\ \hline
       Exact&\multirow{2}{*}{2.85 } & \multirow{2}{*}{-2.5 } & \multirow{2}{*}{1.8} &       2.86 & -2.58 & 1.84 & 0.73\\
       Approx&		& 		&		&$8.44\times10^{3}$ & $-8.15\times10^{6}$ & 2.66  & 0.29\\ \hline
       Exact&\multirow{2}{*}{1.6 } & \multirow{2}{*}{-0.75 } & \multirow{2}{*}{0.1} &       1.61 & -0.75 & 0.11 & 0.70   \\
       Approx&		& 		&		&$1.36\times10^{7} $& $-1.15\times10^{10}$  & 0.37  & $5.12\times10^{-06}$\\
\hline
\end{tabular}
\caption{\new{
Quantitative assessment of the numerical study: sets of true parameters (left), average estimations over 100 repetitions (middle) and average \(p\)-values for the test of Section~\ref{sec:goodness}.
}}
\label{table:estim_pvalue}
\end{table}

\section{Discussion}

In this paper we proposed a maximum likelihood approach for Hawkes processes that can handle both self-exciting and self-regulating scenarios, the first case being already covered in the literature and the latter being our main contribution. For this purpose, we define the concepts of underlying intensity function and restart times when working with monotone kernel functions. In particular we obtain exact expressions of the compensator for the exponential Hawkes process which is the key step of the estimation procedure.
We present numerical results on synthetic data that show the efficiency of our procedure, with a substantial improvement compared to approximated approaches when the intensity function is frequently null.

From a theoretical point of view, future work will consist in adapting analytical results to study the convergence of our estimator in the self-regulating case. Regarding modeling, it would be of great interest to consider kernel functions outside the classical exponential scenario.
Another important step is the extension of our concepts and algorithms to the multivariate version of the process, which is not straightforward since in the multivariate setting the expression of the restart times are no longer explicit.
\new{This last point is essential in order to target real-world datasets since in many applications, being limited to the univariate case will lead to detect self-excitation. However, a model that accounts for potential inhibition effects is of great interest when considering interactions between events of different natures, which will typically be modeled by a multivariate process. This multidimensional extension is the object of a future work, with a further perspective to use our procedure in neuroscience applications in order to detect attraction and repulsion effects between neurons.}

\new{
\section{Acknowledgments}
We thank the Associate Editor and the referees for their valuable comments helping to improve greatly the overall quality of this letter.
}

\bibliography{bibliography}

\pagebreak

\appendix

\section{Proof of Proposition~\ref{prop:exp}}
\label{app:appendix}

Let us begin by expressing the underlying intensity function between two event times.
First, \(\lambda^\star(t) = \lambda_0\) for \(t\in[0,T_1)\).
Then, for any $k\in\NN$, for all $t \in [T_k,T_{k+1})$, the underlying intensity is differentiable in \(t\) and
\begin{equation*}
    (\lambda^\star)'(t) = -\beta(\lambda^\star(t) - \lambda_0),
\end{equation*}
with the left condition: $\lambda_k^\star := \lambda^\star(T_k)$.
Solving this differential equation leads to
\begin{equation}\label{eq:markov_expression}
    \lambda^\star(t) = \lambda_0 + (\lambda_k^\star - \lambda_0)\mathrm{e}^{-\beta(t - T_k)}.
\end{equation}

Now, by definition of the restart time $T_k^\star = \inf{\{t\geq T_k\mid \lambda(t) > 0\}}$, we have that if $\lambda_k^\star \geq 0$, then $T_k^\star = T_k$. Otherwise, as $\lambda^\star$ is continuous on the interval $[T_k,T_{k+1})$, we obtain $T_k^\star$ by solving for $t$: \(\lambda^\star(t) = 0\).
Thus, by Equation~\eqref{eq:markov_expression}:
\begin{align*}
    \lambda^\star(T_k^\star) = 0
    \iff T_k^\star = T_k + \beta^{-1}\log{\left(\frac{\lambda_0 - \lambda^\star_k}{\lambda_0}\right)}.
\end{align*}
Gathering both situations, we obtain the first part of Proposition~\ref{prop:exp}: \(T_k^\star = T_k + \beta^{-1}\log{\left(\frac{\lambda_0 - \lambda^\star_k}{\lambda_0}\right)} \II_{\lambda_k^\star < 0}\).

Let now $k\in\{1,\ldots,N(t)\}$ and $\tau\in[T_{k}^\star,T_{k+1}]$.
By Equation~\eqref{eq:markov_expression},
\begin{align*}
    \int_{T_{k}^\star}^{\tau}{\lambda^\star(u)\,\mathrm{d}u}
    &= \int_{T_{k}^\star}^{\tau} \left( \lambda_0 + (\lambda_{k}^\star - \lambda_0)\mathrm{e}^{-\beta(u - T_{k})} \right)\,\mathrm{d}u\\
    &= \lambda_0(\tau - T_{k}^\star) + \beta^{-1}(\lambda_{k}^\star - \lambda_0)(\mathrm{e}^{-\beta(T_{k}^\star-T_{k})}-\mathrm{e}^{-\beta(\tau-T_{k})}),
\end{align*}
which is the second part of Proposition~\ref{prop:exp}.

\section{Simulation algorithm}
\label{app:simulation}

Algorithm~\ref{alg:simulation} builds upon Ogata's thinning simulation method \cite[Proposition 1]{Ogata1981} in order to handle Hawkes processes with either self-excitation or inhibition.

\begin{algorithm}[!h]
\SetAlgoLined
 \textbf{Input} Parameters $\lambda_0$, $h$ a monotone function, and a stopping criteria (end-time $T$ or maximal number of jumps $N_{max}$)\;
 \textbf{\underline{Initialization}} Initialize $\lambda_k =\lambda_0$, $t_k=0$\ and list of times $\mathcal{T} = \emptyset$\;
 \While{\textnormal{Stopping criteria not fulfilled}}{
 Set $\lambda_{max} = \max(\lambda_0,\lambda_k)$\;
 Generate candidate time $t_{cand} = t_k - \frac{\log(U_1)}{\lambda_{max}}$, $U_1\sim U([0,1])$\;
 Estimate intensity $\lambda_k = \lambda(t)$ using sequence of times $\mathcal{T}$\;
 Sample $U_2\sim U([0,1])$\;
 \If{$U_2 \leq \frac{\lambda_k}{\lambda_{max}}$}{
        Add $t$ to sequence of times $\mathcal{T}$\;
     }
 Set $t_k = t_{cand}$\;
 }
 \Return the sequence of jumps $\mathcal{T}$.
 \caption{Thinning algorithm for monotone Hawkes process.}
 \label{alg:simulation}
\end{algorithm}

\end{document}